\documentclass[11pt,reqno]{amsart}
\pdfoutput=1
\usepackage{etoolbox}
\newcommand{\appendixText}{}

 \newcommand{\sv}[1]{}
 \newcommand{\lv}[1]{#1}
 \newcommand{\toappendix}[1]{#1}

\usepackage{amsmath}
\usepackage{amssymb}
\usepackage[left=2.54cm,top=2.54cm,right=2.54cm,bottom=2.54cm]{geometry}
\usepackage{epsfig}
\usepackage{bm}
\usepackage[utf8]{inputenc}
\usepackage[T1]{fontenc}
\usepackage{algpseudocode}

\usepackage{hyperref}

\usepackage[backend=biber,style=numeric-comp,sorting=nyt]{biblatex}
\newtheorem{theorem}{Theorem}[section]
\newtheorem{corollary}[theorem]{Corollary}

\newtheorem{lemma}[theorem]{Lemma}

\newtheorem{prop}[theorem]{Proposition}

\theoremstyle{definition}
\newtheorem{remark}{Remark}

\usepackage[usenames,dvipsnames]{xcolor}

\newcommand{\remove}[1]{}

\newcommand{\scr}{\mathcal}
\newcommand{\mb}{\mathbb}
\newcommand{\til}{\widetilde}

\newcommand{\whp}{{\it w.h.p.}}
\newcommand{\iid}{{\it i.i.d.}}
\newcommand{\uar}{{\it u.a.r.}}

\lv{}

\DeclareMathOperator{\disc}{disc}

\newcommand{\norm}[1]{\left\lVert #1 \right\rVert}

\newcommand{\nat}{\mathbb{N}}

\newcommand{\Ber}{\textup{Ber}}
\newcommand{\Bin}{\textup{Bin}}

\newcommand{\eps}{\varepsilon}
\newcommand{\ex}{\mathrm{ex}}

\def\Pr{\mathbb{P}}
\def\pr{\mathbb{P}}

\def\ex{\mathbb{E}}
\def\var{\mathbf{Var}}

\def\le{\leqslant}
\def\ge{\geqslant}
\def\leq{\leqslant}

\makeatletter
\def\Ddots{\mathinner{\mkern1mu\raise\p@
\vbox{\kern7\p@\Hbbox{.}}\mkern2mu
\raise4\p@\Hbbox{.}\mkern2mu\raise7\p@\Hbbox{.}\mkern1mu}}
\makeatother

\newcommand{\df}{:=}

\begin{document}


\title{The Phase Transition of Discrepancy in Random Hypergraphs}

\sv{\author{Anonymous}}

\lv{
\author{Calum MacRury}
\address{Department of Computer Science, University of Toronto, Toronto, ON, Canada}
\email{cmacrury@cs.toronto.edu}

\author{Tom\'a\v{s} Masa\v{r}\'ik}
\address{Department of Applied Mathematics, Faculty of Mathematics and Physics, Charles University, Prague, Czech Republic~\&
Faculty of Mathematics, Informatics and Mechanics, University of Warsaw, Warsaw, Poland~\& Department of Mathematics, Simon Fraser University, Burnaby, BC, Canada}
\email{masarik@kam.mff.cuni.cz}

\author{Leilani Pai}
\address{Department of Mathematics, University of Nebraska-Lincoln, Lincoln NE, USA}
\email{lpai@huskers.unl.edu}

\author{Xavier P\'erez-Gim\'enez}
\address{Department of Mathematics, University of Nebraska-Lincoln, Lincoln NE, USA}
\email{xperez@unl.edu}
}

\maketitle

\begin{abstract}
Motivated by the Beck-Fiala conjecture, we study the discrepancy problem in two related models of random hypergraphs on $n$ vertices and $m$ edges. In the first ({\em edge-independent}) model, a random hypergraph $H_1$ is constructed by fixing a parameter $p$ and allowing each of the $n$ vertices to join each of the $m$ edges independently with probability $p$. In the parameter range in which $pn \rightarrow \infty$ and $pm \rightarrow \infty$, we show that with high probability (\whp) $H_1$ has discrepancy at least $\Omega(2^{-n/m} \sqrt{pn})$ when $m = O(n)$, and at least $\Omega(\sqrt{pn \log\gamma })$ when $m \gg n$,
where $\gamma = \min\{ m/n, pn\}$. In the second ({\em edge-dependent}) model, $d$ is fixed and each vertex of $H_2$ independently joins exactly $d$ edges uniformly at random.
We obtain analogous results for this model by generalizing the techniques used for the edge-independent model with $p=d/m$. Namely, for $d \rightarrow \infty$ and $dn/m \rightarrow \infty$, we prove that \whp\ $H_{2}$ has discrepancy at least $\Omega(2^{-n/m} \sqrt{dn/m})$ when $m = O(n)$, and at least $\Omega(\sqrt{(dn/m) \log\gamma})$ when $m \gg n$, where $\gamma =\min\{m/n, dn/m\}$. 
Furthermore, we obtain nearly matching asymptotic upper bounds on the discrepancy in both models (when $p=d/m$), in the dense regime of $m \gg n$. Specifically, we apply the partial colouring lemma of Lovett and Meka to show that \whp\ $H_{1}$ and $H_{2}$ each have discrepancy $O( \sqrt{dn/m}  \log(m/n))$, provided $d \rightarrow \infty$, $d n/m \rightarrow \infty$ and $m \gg n$. This result is algorithmic, and together with the work of Bansal and Meka characterizes how the discrepancy of each random hypergraph model transitions from $\Theta(\sqrt{d})$ to $o(\sqrt{d})$ as $m$ varies from $m=\Theta(n)$ to $m \gg n$.
\end{abstract}

\section{Introduction}
A {\bf hypergraph}\lv{\footnote{The definition of a hypergraph is equivalent to that of a set system, though we exclusively use the former terminology in this work.}} $H=(V,E)$ consists of a set $V=\{v_1,\ldots,v_n\}$ of $n$ vertices together with a multiset $E=\{e_1,\ldots,e_m\}$ of $m$ edges, where each edge $e_i$ is a subset of $V$. We denote the size of an edge $e$ as $|e|$. 
Note that $H$ is allowed to have duplicate edges (i.e.~we may have $e_i=e_{i'}$ for some $i\ne i'$).
We can bijectively represent $H$ by an $m\times n$ $\{0,1\}$-matrix $\bm{A}=\bm{A}(H) = (A_{i,j})_{1\le i\le m,\, 1\le j\le n}$, where $A_{i,j}=1$ if $v_j\in\ e_i$ and $A_{i,j}=0$ if $v_j\notin\ e_i$. We call $\bm A$ the {\bf incidence matrix} of $H$.
In particular, each pair of duplicate edges in $H$ corresponds to a pair of identical rows in $\bm A$.
Moreover, we define the {\bf degree} of a vertex $v_j$ of $H$ to be the number of edges containing that vertex, i.e.~the number of $1$'s in the $j$th column of $\bm A$.
A classical problem in discrepancy theory is to find a 2-colouring of the vertices of $H$ so that every edge is as ``balanced'' as possible.
To make this more precise, we define a \textbf{colouring} of $H$ to be a function $\psi: V \rightarrow \{-1,1\}$. This can be extended to a map $\psi:E\to\mathbb Z$, by defining
$
    \psi(e) := \sum_{v \in e} \psi(v)
$
for each $e\in E$. We call $|\psi(e)|$ the {\bf discrepancy} of edge $e$, and note that it measures how unbalanced the colouring $\psi$ is on that edge.
Further, the \textbf{discrepancy} of colouring $\psi$, denoted $\disc(\psi)$, is the discrepancy of the least balanced edge of $H$. That is, 
\[
    \disc(\psi):= \max_{e \in E}  |\psi(e)|
\]
Finally, we define the \textbf{discrepancy} of hypergraph $H$ as
\[
  \disc(H):= \min_{ \psi} \disc(\psi),
\]
where the minimization is over all colourings of $V$.

\lv{This and other related notions of combinatorial discrepancy have been studied from various angles and in different contexts. (For a more detailed introduction to the subject, we refer to books~\cite{Matousek_1999}, \cite{Chazelle_2000} and~\cite{Chen_2014}.)} 
One of the central problems in discrepancy theory in the above setting is to bound the discrepancy of a hypergraph $H$ in terms of its maximum degree $d$. In~\cite{Beck_1981}, it was proven by Beck and Fiala 
that the discrepancy of $H$ is no larger than $2d-1$. Moreover, they conjectured that the correct
upper bound is of the order $O(d^{1/2})$. 
%
There has been much work in trying to improve on the original bound of~\cite{Beck_1981}. Most recently, it was proven by Bukh~\cite{Bukh_2016} that $\disc(H) \le 2d - \lg^{*}(d)$, where $\lg^*$ is
the binary iterated logarithm. This of course yields no asymptotic improvement in terms of $d$, but to this date is the strongest upper bound known which solely depends on $d$.
If the upper bound is allowed dependence on the multiple parameters of the hypergraph, then there are results yielding improvements for hypergraphs in the correct range of parameters. For instance, Banaszczyk~\cite{Banaszczyk1998} showed that if $n:=|V|$,
then $\disc(H) = O( \sqrt{d \log n})$---a bound which was later made algorithmic by Bansal and Meka~\cite{Bansal_2018}.
Recently, Potukuchi~\cite{Potukuchi_2019} proved that, for $d$-regular $H$, $\disc(H) = O(\sqrt{d} + \lambda)$, where $\lambda := \max_{v \perp \bm{1},  \norm{v}_{2} =1} \norm{\bm{A} \, v}_{2}$ and
$\bm A$ is the incidence matrix of $H$.

In order to find upper bounds which depend solely on the maximum degree, restricted classes of hypergraphs have instead been studied. For example, if $H=(V,E)$ is assumed to
be both $d$-regular and $d$-uniform (that is, the incidence matrix $\bm A$ has exactly $d$ $1$'s in each column and row), then a folklore application of the Lov\'{a}sz local lemma can be used to show that there exists a colouring which achieves discrepancy $O( \sqrt{d \log d})$ (see~\cite{Ezra_2015, Potukuchi_2018} for details). 
\lv{%

}%
Another approach is to restrict one's attention to hypergraphs which are generated
randomly. In this work, we focus on two specific random hypergraph models. Both of these models
are defined as distributions over the set of hypergraphs with $n \ge 1$ vertices and
$m \ge 1$ edges. 



\subsection{The Edge-Independent Model}

In~\cite{Hoberg_2019}, Hoberg and Rothvoss introduced a random hypergraph model,
denoted $\mb{H}(n,m,p)$,
in which a probability parameter $0 \le  p \le 1$
is given (in addition to $n$ and $m$).
Their model, which we refer to as the
\textbf{edge-independent model}, is a distribution
on hypergraphs which we describe through the following randomized procedure:

\begin{itemize}
    \item Fix the vertex set $V=\{v_1, \ldots, v_n\}$.
    \item For each $1\leq i \leq m$, construct edge $e_i$ by placing each $v\in V$ in $e_i$ independently with probability $p$.
\end{itemize}


We denote $E=\{e_{1}, \ldots ,e_{m}\}$
and define $\mb{H}(n,m,p)$ to be the distribution of the random hypergraph $H=(V,E)$.
In other words, the entries of the incidence matrix $\bm A$ of $H$ are independent Bernoulli random variables of parameter $p$.
We write $H \sim \mb{H}(n,m,p)$
to indicate that $H$ is drawn from $\mb{H}(n,m,p)$.

If $m=m(n)$ and $p=p(n)$ are functions which depend on $n$,
then we say that $\mb{H}(n,m,p)$ satisfies a property  $Q=Q(n)$ \whp, provided that
$\mb{P}[\text{$H(n)$ satisfies $Q(n)$}] \rightarrow 1$ as $n \rightarrow \infty$,
where $H =H(n)$ is drawn from $\mb{H}(n,m,p)$. Often, we abuse terminology slightly and say that
the random hypergraph $H$ satisfies $Q$ \whp


Hoberg and Rothvoss showed that, if $n \ge C_1 m^{2} \log m$ and $C_1  \log{n} /m \le p \le 1/2$ 
for some sufficiently large constant $C_1 >0$,
then $\disc(H) \le 1$ \whp\ for $H \sim \mb{H}(n,m,p)$. 
A natural question left open by their work is whether or not
$H$ continues to have constant discrepancy when $n$ transitions from $C_1 m^{2} \log m$
to $\Theta(m \log m)$. Potukuchi~\cite{Potukuchi_2018} provided a positive answer to
this question for the special case when $p=1/2$ by showing that if $n \ge C_2 m \log m$ for $C_2 = (2 \log 2)^{-1}$, then \whp\ $\disc(H) \le 1$. 
Very recently, Altschuler and Niles-Weed~\cite{Altschuler2021} used
Stein's method~\cite{Holmes} in conjunction with the second moment method
to substantially generalize this result to hold when $p =p(n)$ depends on $n$.
This includes the challenging case when $p(n) \rightarrow 0$.

\subsection{The Edge-Dependent Model}
A related model was introduced by Ezra and Lovett in~\cite{Ezra_2015}.
As before, we fix $n \ge 1$ and $m \ge 1$, yet we now also consider a parameter $d \ge 1$ which satisfies $d \le m$. 
The \textbf{edge-dependent model}, denoted $\scr{H}(n,m,d)$, is again a distribution
on hypergraphs, though we describe it through a different randomized procedure:

\begin{itemize}
    \item Fix vertex set $V=\{v_1, \ldots, v_n\}$.
    \item For each vertex $v\in V$, independently and \uar\ (uniformly at random) draw $I_{v} \subseteq [m] $ with $|I_v| =d$.
    \item For each $1\leq i \leq m$, construct edge $e_i$ by defining $e_i:= \{u\in V: i\in I_u\}$.
\end{itemize}

By setting $E:=\{e_{1}, \ldots ,e_{m}\}$,
we define $\scr{H}(n,m,d)$ to be the distribution of the random hypergraph $H=(V,E)$.
In other words, the incidence matrix $\bm A$ of $H\sim\scr{H}(n,m,d)$ is a random $m\times n$ matrix where each column has $d$ ones and $m-d$ zeros. Note that the columns of $\bm A$ are independent, but the rows are not.
We define what it means
for $\scr{H}(n,m,d)$ to satisfy a property \whp\ in the same way as in the edge-independent model.

Ezra and Lovett showed that, if $m \ge n \ge d \to\infty$, then \whp\ $\disc(H) = O(\sqrt{d \log d})$. Bansal and Meka~\cite{Bansal_2018} later showed that the factor of $\sqrt{\log d}$ is redundant, thereby matching the bound claimed in the Beck-Fiala
conjecture. Specifically, they showed that, for the entire range of $n$ and $m$,
$\disc(H) = O( \sqrt{d})$ \whp, provided $d = \Omega( (\log \log m)^{2})$. 
This result can be easily modified to also apply to the edge-independent model, provided the analogous condition $pm = \Omega( (\log \log m )^{2})$ holds.

In~\cite{Cole2020}, Franks and Saks considered the more general problem of \textbf{vector balancing}.
Their main result concerns a collection of random matrix models in which the columns are generated independently.
In particular, their results apply to the sparse regime ($m \ll n$) of both the random hypergraph models we have discussed. Specifically, they show that if $n = \Omega( m^{3} \log^{3} m)$, then \whp\ $\disc(H) \le 2$,
provided $H$ is drawn from $\mb{H}(n,m,p)$ or $\scr{H}(n,m,d)$ for $p =d/m$.
Finally, in a very recent work, Turner et al.~\cite{Turner2020} considered the problem of
vector balancing when the entries of the random matrix $\bm{A}$ are each distributed
as standard Gaussian random variables which are independent and identically distributed (\iid).
Amongst other results, they showed that the discrepancy of $\bm{A}$ is $\Theta(2^{-n/m}  \sqrt{n})$ \whp,
provided $m \ll n$.


\subsection{An Overview of Our Results}

All results proven in this paper are asymptotic with
respect to the parameter $n$, the number of vertices of the model. 
Thus, we hereby assume that $m=m(n)$, $p =p(n)$ and $d=d(n)$ are functions which depend on $n$, with $p =d/m$.

While previous results have successfully matched (or improved upon) the conjectured Beck-Fiala bound of $\sqrt{d}$  in
the random hypergraph setting, they either apply to a restricted parameter range~\cite{Ezra_2015, Potukuchi_2018, Hoberg_2019, Cole2020},
or do not provide asymptotically tight results for the full parameter range~\cite{Ezra_2015, Bansal_2018}. In particular,
when $n /\log n \ll m \ll n$ or $m \gg n$, the correct order of the discrepancy is unknown in either model.
In this paper, we obtain (almost) matching upper and lower bounds in the dense regime of $m \gg n$. Moreover,
our upper bounds are algorithmic. In the sparse regime of $m \ll n$, we provide the first lower bounds which apply
to the full parameter range under the mild assumption that both the average edge size $dn/m = pn$
and degree $d = pm$ tend to infinity with $n$. Proving the existence of a colouring whose discrepancy matches our lower bound in the regime $n/\log n \ll m \ll n$
remains open. This problem is particularly challenging from an algorithmic perspective, as the partial colouring
lemma~\cite{Lovett_2012} does not appear to be useful in this range of parameters, and this is the main tool used in the literature.

We now formally state our main results:

\begin{theorem} \label{thm:edge_independent_lower_bound} 
Suppose that $H$ is generated from $\mb{H}(n,m,p)$ with $pn\rightarrow \infty$, $pm \rightarrow \infty$ and $p$ bounded away from $1$. If $m = O(n)$, then \whp
\[
    \disc(H) = \Omega\left(\max\{2^{-n/m} \sqrt{pn}, 1\} \right),
\]
Moreover, if $m \gg n$, then \whp
\[
    \disc(H) = \Omega\left(\sqrt{pn\log\gamma} \right),
\]
where $\gamma:= \min\{m/n, pn\}$.
\end{theorem}
\begin{remark}
This bound complements the upper bound of $1$ in~\cite{Altschuler2021} for $m \leq Cn/\log n$ where $C = 2 \log 2$, but also implies that
if $\eps >0$ is a constant, then $H$ has non-constant discrepancy for $m \ge (C + \eps) n/\log (np)$. Thus, $H$ exhibits a sharp phase transition at $Cn/\log(n p)$
for constant $p$.
\end{remark}
We obtain analogous results regarding the edge-dependent model $\scr{H}(n,m,d)$:
\begin{theorem}\label{thm:edge_dependent_lower_bound}
Suppose that $H$ is generated from $\scr{H}(n,m,d)$ with $dn/m \rightarrow \infty$, $d \rightarrow \infty$, and $d/m$ bounded away from $1$. If $m= O(n)$, then \whp
\[
    \disc(H) = \Omega\left(\max\{2^{-n/m} \sqrt{\frac{dn}{m}}, 1\} \right),
\]
Moreover, if $m \gg n$, then \whp
\[
\disc(H) = \Omega\left(\sqrt{\frac{dn}{m} \log\gamma} \right),
\]
where $\gamma:= \min\{m/n, dn/m\}$.
\end{theorem}
\begin{remark} \label{remark:matching_lower_bound}
The techniques used in~\cite{Altschuler2021} do not seem to apply to the edge dependent model.
Thus, if $m = O(n^{1/3}/\log n)$, then~\cite{Cole2020} implies $\disc(H) \le 2$, however
when $n^{1/3}/\log n \ll m \ll n$, $O(\sqrt{d})$ remains the best known upper bound for $\disc(H)$~\cite{Bansal_2018}.
\end{remark}

Finally, we prove an upper bound in the dense regime of $m \gg n$ which holds
for both models:
\begin{theorem} \label{thm:dense_setting}
Assume that $H$ is drawn from $\scr{H}(n,m,d)$ or $\mb{H}(n,m,p)$ with $p=d/m$, and pick any
$\beta=\beta(n)\ge 1$ satisfying\footnote{%
Let us remark at this place that by $\log$ we always mean the natural logarithm. 
In the proof of this theorem it is convenient to use $\lg$ to denote logarithm of base $2$. 
}
\begin{equation}\label{eqn:beta_asymptotics}
    \beta \frac{dn}{m} \ge \log(m/n)  \left(\log\left(\frac{d  n}{m} \right) + 2 \right)^{5}.
  \end{equation}
If $m \gg n$, and $d n/m \rightarrow \infty$, then 
\whp
\[
  \disc(H) = O\left( \sqrt{\frac{d n}{m}\, \log\left(\frac{m}{n} \right)\,\beta}  \right).
\]
Moreover, whenever this holds, we can find a colouring $\psi$ of $H$ with such a discrepancy
in expected polynomial time. 
\end{theorem}
\begin{remark}
Observe that the smaller we are able to take $\beta$, the better upper bound we get.
In particular, if $\beta:= \log(m/n)$, then \eqref{eqn:beta_asymptotics} and $\beta(n)\ge 1$ are satisfied (for large enough $n$).
Therefore, at worst the upper bound is $O\left( \sqrt{\frac{d n}{m}} \log\left(\frac{m}{n} \right) \right)$, which is significantly
smaller than the upper bound of $O(\sqrt{d})$ of~\cite{Bansal_2018}, as $m \gg n$.
\end{remark}


Theorem~\ref{thm:dense_setting} provides asymptotically matching bounds for the lower bounds of Theorems~\ref{thm:edge_independent_lower_bound} and~\ref{thm:edge_dependent_lower_bound}  in 
a broad range of the dense regime $m \gg n$.
For instance, this happens when $d = p  m  \ge (m/n)^{1+\eps}$, where $\eps >0 $ is a constant, since in that case $\beta:=1$ clearly satisfies~\eqref{eqn:beta_asymptotics},
and $\log (dn/m) = \Omega(\log(m/n))$, so $\gamma = \Theta(\log(m/n))$.

\begin{corollary} \label{cor:matching}
Assume that $H$ is drawn from $\scr{H}(n,m,d)$ or $\mb{H}(n,m,p)$, where $p=d/m$.
If $m \gg n$ and $d = pm  \ge (m/n)^{1+\eps}$ for constant $\eps>0$, then 
\whp
\[
    \disc(H) = \Theta\left( \sqrt{\frac{d n}{m}  \log\left(\frac{m}{n} \right) } \right).
\]
\end{corollary}
\begin{remark}
This shows that in the dense regime, 
the main parameter of interest describing the discrepancy of each random hypergraph model is the average edge
size, $dn/m$, opposed to $d$, the average/maximum degree (depending on the model).
\end{remark}

%

\section{Preliminaries} \label{sec:prelim}

\sv{\toappendix{\section{Additions to Section~\ref{sec:prelim}}\label{app:prelim}}}


In this section, we first provide some central-limit-type results for sums of independent random variables, which will be used in Section~\ref{sec:lower} to obtain lower bounds on the discrepancy.
In the sequel, we write $N(0,1)$ to denote a generic random variable distributed as a standard Gaussian with cumulative distribution function
\[
    \Phi(x) := \pr[N(0,1)\le x] = \frac{1}{\sqrt{2  \pi}} \int_{- \infty}^{x} \exp(-t^{2}/2) \, dt
    \qquad\text{for each }x\in\mb{R}.
\]
The Berry-Esseen Theorem (see section~XVI.5 in~\cite{Feller-II}), which we state below for convenience, yields a quantitative form of the Central Limit Theorem for sums of independent random variables with finite third moment. 
\begin{theorem}[Berry-Esseen]\label{thm:berry_esseen}
  The following holds for some universal constant\sv{\footnote{%
There has been a series of works improving
upon the constant $c_{uni}$, the latest of which by Shevtsova~\cite{Shevtsova2010} shows
that $c_{uni}/2 \le 0.560$ in our setting. However, since we are only concerned with the asymptotic growth of discrepancy, the precise value of $c_{uni}$ is not important.}} $c_{uni} > 0$.
  Let $Y_{1},\ldots, Y_{n}$ be independent random variables with $\mb{E}[Y_{i}]=0$, $\mb{E}[Y^2_{i}]=\sigma^2_i$ and  $\mb{E}[|Y_{i}|^3]=\rho_i<\infty$ for all $i\in[n]$.
Consider the sum $Y = \sum_{i=1}^{n} Y_{i}$, with standard deviation $\sigma = \sqrt{\sum_{i=1}^{n} \sigma_i^2}$, and let $\rho = \sum_{i=1}^{n} \rho_{i}$.
Assume $\sigma>0$. Then,
\[
    \sup_{x\in\mb{R}} | \mb{P}[ Y/\sigma \le x] - \Phi(x)| \le  \frac{(c_{uni}/2)\rho}{\sigma^3}.
\]
\end{theorem}
\lv{
\begin{remark}
There has been a series of works improving
upon the constant $c_{uni}$, the latest of which by Shevtsova~\cite{Shevtsova2010} shows
that $c_{uni}/2 \le 0.560$ in our setting. However, since we are only concerned with the asymptotic growth of discrepancy, the precise value of $c_{uni}$ is not important.
\end{remark}
}
Theorem~\ref{thm:berry_esseen} and the triangle inequality immediately yield the following corollary:
\begin{corollary} \label{cor:berry_interval_setting}
For any interval $I \subseteq (-\infty, \infty)$,
\[
    |\mb{P}[ Y/\sigma \in I] - \mb{P}[ N(0,1) \in I]| \le \frac{c_{uni}  \rho}{\sigma^3}.
\]
\end{corollary}
We will apply this result to linear combinations of independent Bernoulli's with coefficients in $\{-1,1\}$.
More precisely, let $X_{1},\ldots,X_{n}$ be independent random variables with $X_{i} \sim \Ber(p_{i})$ for some $\bm p = (p_1,\ldots,p_n) \in [0,1]^n$
(where $\Ber(p_i)$ is a Bernoulli of parameter $p_i$).
Given a vector $\bm{a}=(a_{1}, \ldots , a_{n}) \in \{-1,1\}^{n}$, consider the sum
$S_{\bm a,\bm p} := \sum_{k=1}^{n} a_{i}  X_{i}$, whose
standard deviation we denote by $\sigma$. Under these assumptions we obtain the following bound.
\begin{lemma} \label{lem:general_interval_probability}
For any bounded interval $[L,R] \subseteq (-\infty, \infty)$,
\[
 \mb{P}[  S_{\bm a,\bm p}  \in [L,R]] \le \frac{c_{uni}}{\sigma} + \left(1 - \exp\left(-\frac{(R-L)^{2}}{2  \pi  \sigma^{2}} \right) \right)^{1/2}.
\]
\end{lemma}
(When $\sigma=0$, the right-hand side of the bound above is simply interpreted as $+\infty$.)
\sv{The proof of the lemma is postponed to Appendix~\ref{app:prelim}.}
\toappendix{%
\sv{\begin{proof}[Proof of Lemma~\ref{lem:general_interval_probability}]}
\lv{\begin{proof}}
Let $\mu = \mb{E}[S_{\bm a,\bm p}] = \sum_{i=1}^n a_ip_i$.
Then we centre $S_{\bm a,\bm p}$ by defining the random variables $Y_{i} = a_{i}(X_{i}-\mb{E}[X_{i}])$ and setting
\[
    Y:= \sum_{k=1}^{n} Y_{i} = S_{\bm a,\bm p} - \mu.
\]
Observe that $Y$ has the same standard deviation $\sigma$ as $S_{\bm a,\bm p}$, which we assume is non-zero (otherwise the lemma holds trivially). Moreover, $S_{\bm a,\bm p} \in [L,R]$ if and only if $Y/\sigma \in [\til{L}, \til{R}]$,
where  $\til{L}:=(L -\mu)/\sigma$ and $\til{R}:= (R-\mu)/\sigma$.
Further,
\[
\mb{E}[|Y_{i}|^3] = (1-p_{i})p_{i}^3 + p_{i}(1-p_{i})^3 = p_{i}(1-p_{i}) \left( p_{i}^2 + (1-p_{i})^2 \right) \le p_{i}(1-p_{i}),
\]
and hence
\[
\rho = \sum_{i=1}^n \mb{E}[|Y_{i}|^3] \le \sum_{i=1}^n p_{i}(1-p_{i}) = \sigma^2.
\]
Then, Corollary~\ref{cor:berry_interval_setting} immediately yields
\[
\mb{P}[ S_{\bm a,\bm p} \in [L,R]] = \mb{P}[ Y/\sigma \in [\til{L}, \til{R}]]
\le \frac{c_{uni}}{\sigma} + \mb{P}[ N(0,1) \in [\til{L}, \til{R}]].
\]
To finalize the proof, it only remains to bound $\mb{P}[ N(0,1) \in [\til{L}, \til{R}]]$. In order to do so, we will use the inequality
\begin{equation}\label{eqn:error_function}
    \int_{-t}^{t} \exp(-x^{2}/2) \, dx \le
    \sqrt{2\pi  ( 1 - \exp(-2t^2/\pi) ) }
    \qquad\text{for all } t\in\mb R,
\end{equation}
which can be found in~\cite{Williams46}.
Furthermore, note that $\exp(-x^{2}/2)$ is an even function, decreasing with $x^2$.
Combining this fact with~\eqref{eqn:error_function}, we get
\begin{align*}
    \mb{P}[ N(0,1) \in [\til{L}, \til{R}]] &= \frac{1}{ \sqrt{2 \pi}} \int_{\til{L}}^{\til{R}}\exp( -x^{2}/2) \, dx
    \\
    &\le \frac{1}{ \sqrt{2 \pi}} \int_{-\left(\til{R}-\til{L}\right)/2}^{(\til{R} - \til{L})/2} \exp( -x^{2}/2) \, dx
    \\
    &= \frac{1}{ \sqrt{2 \pi}} \int_{-(R-L)/2\sigma}^{(R-L)/2\sigma} \exp( -x^{2}/2) \, dx
    \\
    &\le  \sqrt{ 1 - \exp\left(-\frac{(R-L)^2}{2\pi\sigma^2}\right) },
\end{align*}
which concludes the proof of the lemma.
\end{proof}
}%
Now suppose there exist $0 < p < 1$, $0 < \zeta < 1$ and $0 \le  \eps < 1$ such that
\begin{equation} \label{eqn:averages_approximation}
\sum_{i=1}^{n} p_{i} \ge   (1 - \eps) p n
\qquad\text{and}\qquad
p_{i} \le \zeta \quad \text{for each } i\in[n].
\end{equation}
Then we can restate the upper bound of Lemma~\ref{lem:general_interval_probability} in the following convenient
way, which we use as our key tool in proving Theorems~\ref{thm:edge_independent_lower_bound} and~\ref{thm:edge_dependent_lower_bound}.
\sv{The proof is deferred to Appendix~\ref{app:prelim}.}
\begin{lemma} \label{lem:low_discrepancy_probability}
Suppose $p_{1}, \ldots , p_{n}$ satisfy \eqref{eqn:averages_approximation} for some $0 \le \eps < 1$,
$0 < p < 1$ and $0 < \zeta < 1$. Then, for any bounded interval $[L,R] \subseteq (-\infty, \infty)$,
\begin{align}
    \mb{P}[ S_{\bm a,\bm p} \in [L,R]] &\le
    \frac{c_{uni}}{\sqrt{(1- \zeta)(1 - \eps) n  p}} + \left(1 - \exp\left(-\frac{(R-L)^{2}}{2\pi (1- \zeta)(1 - \eps)np} \right) \right)^{1/2}
    \label{eqn:low_discrepancy_first_bound}
    \\
    &\le  \frac{c_{uni} + |R-L|/\sqrt{2\pi}}{\sqrt{(1- \zeta)(1 - \eps) n  p}} \label{eqn:low_discrepancy_rough_bound}.
\end{align}
\end{lemma}

\toappendix{
  \sv{\begin{proof}[Proof of Lemma~\ref{lem:low_discrepancy_probability}]}
  \lv{\begin{proof}}
First note that the variance $\sigma^2$ of $S_{\bm a,\bm p}$ satisfies
\[
\sigma^2 = \sum_{i=1}^n p_i(1-p_i) \ge (1-\zeta)\sum_{i=1}^n p_i \ge (1-\zeta)(1-\eps)pn.
\]
This bound, used with Lemma~\ref{lem:general_interval_probability}, immediately gives~\eqref{eqn:low_discrepancy_first_bound}. 
Then~\eqref{eqn:low_discrepancy_rough_bound} follows by applying the inequality $1-\exp(-x) \le x$, which holds for every $x\in \mb{R}$.
\end{proof}
}


\lv{In Theorem~\ref{thm:dense_setting}, we prove an upper bound on discrepancy in the dense regime ($m  \gg n$).
In this parameter range, we make use of the \textbf{algorithmic partial colouring lemma},
a seminal result of Lovett and Meka~\cite{Lovett_2012}
later made deterministic by Levy, Ramadas, and Rothvoss~\cite{Levy2017DeterministicDM}.
We defer the statement of
this result to Lemma~\ref{lem:partial_colouring} of Section~\ref{sec:upper-bound},
as it will not be needed until then.

}




\section{Lower Bounding Discrepancy}\label{sec:lower}

\subsection{The Edge-Independent Model}
We now return to the setting of hypergraph discrepancy in the 
context of the edge-independent model $\mb{H}(n,m,p)$. Throughout this section, $m=m(n)$, $p=p(n)$ and asymptotic statements are with respect to $n\to\infty$.
We first observe that \whp\ there are some edges containing an odd number of vertices and thus the discrepancy cannot be zero. \sv{The proof is included
in Appendix \ref{app:lower}.}
\begin{prop} \label{prop:independent_constant_regime}
Suppose $H\sim\mb{H}(n,m,p)$ with $m\to\infty$, $pn\to\infty$ and $p$ bounded away from $1$ as $n\to\infty$. Then \whp\ $\disc(H) \ge 1$.
\end{prop}
\sv{\toappendix{\section{Additions to Section~\ref{sec:lower}}\label{app:lower}}}
\toappendix{
\sv{\begin{proof}[Proof of Proposition \ref{prop:independent_constant_regime}]|}
\lv{\begin{proof}}
By hypothesis, $p \le 1-\eps$ for some sufficiently small constant $\eps>0$.
Let $e_{1}, \ldots , e_{m}$ be the edges of $H$, and observe that the number of vertices contained in each edge is distributed as $\Bin(n,p)$. Then the probability that $e_i$ has an even number of vertices is
\begin{equation}\label{eq:even}
\sum_{j\text{ even}} \binom{n}{j}p^j(1-p)^{n-j} =  \frac{1}{2} (1 + (1-2 p)^{n}) = 1/2+o(1),
\end{equation}
where we used the fact that $|1-2p|^n \le \max \{ (1-2\eps)^n,  e^{-2pn} \} = o(1)$ as $n\to\infty$.
Hence, the probability all the edges of $H$ contain an even number of vertices is $(1/2+o(1))^{m}=o(1)$.
Therefore, \whp\ $H$ has an edge with an odd number of vertices, and thus has discrepancy at least $1$.
\end{proof}
}
Proposition \ref{prop:independent_constant_regime} trivially implies Theorem~\ref{thm:edge_independent_lower_bound} in the regime in which $2^{-n/m} \sqrt{pn} = O(1)$.
We now use Lemma~\ref{lem:low_discrepancy_probability} to prove the
remaining cases of Theorem~\ref{thm:edge_independent_lower_bound} via a simple
first moment argument:
\begin{proof}[Proof of Theorem~\ref{thm:edge_independent_lower_bound}]
Suppose that $H=(V,E)$ is generated from $\mb{H}(n,m,p)$ with $pn\to\infty$, $pm\to\infty$ and $p$ bounded away from $1$, as $n\to\infty$.
We define
\[
\hat f = \hat f(n) = \begin{cases}
2^{-n/m} \sqrt{p(1-p)n} & \text{if }m = O(n),
\\
\sqrt{p(1-p) n \log \gamma} & \text{if }m \gg n,
\end{cases}
\]
where $\gamma= \min\{ p  n, m/n\}$, and choose a sufficiently small constant $\kappa>0$.
To prove the theorem, it suffices to show that \whp\ $\disc(H) \ge \max \{ \kappa \hat f, 1\}$.
Note that this is trivially true when $\hat f \le 1/\kappa$, in view of Proposition~\ref{prop:independent_constant_regime}. So we will assume that $\hat f > 1/\kappa$, and show that \whp\ $\disc(H) \ge  \kappa \hat f$.
Let $\Psi$ be the set of all colourings $\psi:V\to\{-1,1\}$, and let $Z$ denote the number of colourings $\psi \in \Psi$ with discrepancy $\disc(\psi) \le \kappa \hat f$.
Since the random edges $e_1,\ldots,e_m$ of $H$ are \iid,
\begin{equation}
\mb{E}[Z] = \sum_{\psi\in\Psi} \Pr[ \disc(\psi) \le \kappa  \hat{f} ] =  \sum_{\psi\in\Psi} \Pr[ |\psi(e_1)| \le \kappa  \hat{f}]^{m}.
\label{eq:EZ}
\end{equation}
Note that $\psi(e_1) = \sum_{i=1}^n \psi(v_i) \bm 1_{v_i\in e_1}$, where $\bm 1_{v_i\in e_1}$ denotes the indicator random variable of the event that edge $e_1$ contains vertex $v_i$, so $\psi(e_1)$ is distributed as $S_{\bm a,\bm p}$ in Section~\ref{sec:prelim} (with $a_i=\psi(v_i)$ and $p_i=\Pr[v_i\in e_1]$).
Hence, by applying~\eqref{eqn:low_discrepancy_rough_bound} in Lemma~\ref{lem:low_discrepancy_probability} (with $\epsilon = 0$, $\zeta = p$ and $[L,R]=[-\kappa\hat f,\kappa\hat f]$) to each one of the $2^n$ terms of the last sum in~\eqref{eq:EZ}, it follows that
\[
\mb{E}[Z] \le 2^n \left( \frac{c_{uni} + 2\kappa\hat f/\sqrt{2\pi}}{\sqrt{p(1- p)n}} \right)^m
 \le 2^n \left( \frac{\kappa\hat f ( c_{uni} + \sqrt{2/\pi})}{\sqrt{p(1- p)n}} \right)^m,
\]
where we also used that $\kappa \hat f>1$. Let us consider first the case that $m=O(n)$. Then, from the definition of $\hat f$ and assuming $\kappa < 1/(c_{uni} + \sqrt{2/\pi})$,
\[
\mb{E}[Z]  \le  \left(  \kappa   (c_{uni} + \sqrt{2/\pi})  \right)^m = o(1).
\]
Now suppose that $m\gg n$. In this case, we bound the factor $\Pr[ |\psi(e_1)| \le \kappa  \hat{f}]$ on the right-hand side of~\eqref{eq:EZ} using the tighter inequality~\eqref{eqn:low_discrepancy_first_bound} in Lemma~\ref{lem:low_discrepancy_probability} instead of~\eqref{eqn:low_discrepancy_rough_bound}. Then, assuming that $C:=2\kappa^2/\pi<1/2$, we get
\begin{align*}
\mb{E}[Z] &\le
2^n \left( \frac{c_{uni}}{\sqrt{p(1- p)n}} + \left(1 - \exp\left(-\frac{(2\kappa\hat f)^{2}}{2\pi p(1- p)n} \right) \right)^{1/2} \right)^m
\\
&= 2^n \left( \frac{c_{uni}}{\sqrt{p(1- p)n}} + \left(1 - \gamma^{- 2\kappa^2/\pi} \right)^{1/2} \right)^m
\\
&= \left( 1 + O\left(n/m\right) \right)^m  \left( O(1/\sqrt{pn}) + \left(1 - \gamma^{-C} \right)^{1/2} \right)^m
\\
&= \left( 1 - \frac{1}{2}\gamma^{-C}(1+o(1))  \right)^m
\lv{\\}
\lv{&}= \exp \left(  - \frac{m}{2}\gamma^{-C}(1+o(1))  \right) = o(1),
\end{align*}
where we used the facts $1/\sqrt{pn}=o(\gamma^{-C})$, $n/m=o(\gamma^{-C})$ and $m/\gamma^C\to\infty$.
In either case, $\mb{E}[Z]=o(1)$, and therefore \whp\ $\disc(H) \ge  \kappa \hat f$.
\end{proof}

\subsection{The Edge-Dependent Model}\label{sec:edge-dependent}

In this section, we derive a lower bound on the discrepancy of a hypergraph generated from
the edge-dependent model and prove Theorem~\ref{thm:edge_dependent_lower_bound}.
In view of our previous results for the edge-independent model, one natural approach is to compare both models $\scr{H}(n,m,d)$ and $\mb{H}(n,m,p)$ via a \textbf{coupling procedure}.
For instance, let $m=n$ and $d\gg\log n$ for simplicity, and suppose that we can generate $(H_1,H_2)$ with $H_1\sim\scr{H}(n,n,d)$ and $H_2\sim\mb{H}(n,n,p)$, with edge sets $E(H_1) = \{e^1_1,\ldots,e^1_n\}$ and $E(H_2)=\{e^2_1,\ldots,e^2_n\}$, in such a way that \whp\ for every $i=1,\ldots,n$ we have $|e^1_i\Delta e^2_i|\le\eta$, for some suitable $\eta=\eta(n)$. In particular, this implies that \whp\ $|\disc(H_1) - \disc(H_2)| \le \eta$, and thus $\disc(H_1) = \Omega(\sqrt d + \eta)$ by Theorem~\ref{thm:edge_independent_lower_bound}.
Unfortunately, since the standard deviation of the size of an edge is $\Theta(\sqrt d)$ in either model,
most {\em na\"\i ve} attempts to build such a coupling require $\eta\gg\sqrt d$, which is too large for our purposes. (In fact, it is not hard to build such a coupling with any $\eta\gg\sqrt{d\log n}$.) 
As a result, while it is conceivable that a more delicate coupling argument works, we abandon this approach. Instead, we handle the dependencies of the edges of $H_2$ by applying a careful conditioning argument, while generalizing how we apply Lemma~\ref{lem:low_discrepancy_probability}.

As in the edge-independent model, we first prove a constant lower bound on the discrepancy of $H\sim\scr{H}(n,m,d)$.
\sv{The proof of the following result is deferred to Appendix~\ref{app:lower}.}
\begin{prop}\label{prop:edge_dependent_lower_bound_constant}
Suppose $H\sim\scr{H}(n,m,d)$ with $m \rightarrow \infty$, $dn/m\to\infty$ and $d/m$ bounded away from $1$ as $n\to\infty$. Then \whp\ $\disc(H) \ge 1$.
\end{prop}

\toappendix{
\sv{\begin{proof}[Proof of Proposition~\ref{prop:edge_dependent_lower_bound_constant}]}
\lv{\begin{proof}}
By hypothesis, $p =d/m \le 1-\eps$ for some sufficiently small constant $\eps>0$.
Our goal is to show that \whp\ there is some edge in $H$ containing an odd number of vertices, so $\disc(H)\ge1$.
As in the proof of Proposition~\ref{prop:independent_constant_regime},
the probability that an edge $e_i$ has an odd number of vertices is $1/2+o(1)$ (see~\eqref{eq:even}). Therefore, the expected number $W$ of edges with an odd number of vertices is $(1+o(1))m/2$. Similarly, the probability that two different edges (say, $e_1$ and $e_2$) have an odd number of vertices is $1/4+o(1)$. To see this, first define the bivariate generating function
\[
f(x,y) = \left( \frac{d(d-1)}{m(m-1)}xy + \frac{d(m-d)}{m(m-1)}(x+y) + \frac{(m-d)(m-d-1)}{m(m-1)} \right)^n.
\]
For $0\le i,j\le n$, the coefficient
\[
[x^iy^j] f(x,y)
\]
gives the probability that edge $e_1$ has $i$ ones and edge $e_2$ has $j$ ones. Hence, the probability that both edges have an odd number of vertices is
\[
\frac{f(1,1)-f(-1,1)-f(1,-1)+f(-1,-1)}{4},
\]
which after tedious but straightforward calculations and using the facts that $d/m\le 1-\eps$ and $dn/m\to\infty$, is equal to
\[
\frac{ 1 - 2 \left(1-\frac{2d}{m}\right)^n + \left( 1 - 4\frac{d(m-d)}{m(m-1)} \right)^n }{4}
= 1/4 +o(1),
\]
and hence $\ex W^2 \sim m^2/4$. This implies that $\var W = o((\ex W)^2)$ and, by a standard application of Chebyshev's inequality, we conclude that \whp\ $W\sim m/2\to\infty$.
This proves the proposition.
\end{proof}
}

\lv{\begin{remark}
It is conceivable that in the proof of the proposition above one could obtain an upper bound on $\mb{P}[W=0]$ that is exponentially small in $m$, in the same spirit as in the proof of Proposition~\ref{prop:independent_constant_regime}.
However, this would require some additional work due to the fact that the edges of $H\sim\scr H(n,m,d)$
are \textit{not} formed independently.
\end{remark}}

Proposition~\ref{prop:edge_dependent_lower_bound_constant} trivially implies Theorem~\ref{thm:edge_dependent_lower_bound} in the regime in which $2^{-n/m} \sqrt{dn/m} = O(1)$.
To prove the remaining cases, we will generalize the ideas we used in the proof of
Theorem~\ref{thm:edge_independent_lower_bound}. However, the dependencies among the edges make the argument much more delicate.

\begin{proof}[Proof of Theorem~\ref{thm:edge_dependent_lower_bound}]
Suppose that $H=(V,E)$ is generated from $\scr{H}(n,m,d)$ with $m=m(n)$ and $d=d(n)$ satisfying $dn/m\to\infty$ and $d\to\infty$ (as $n\to\infty$) and with $p=d/m \le c$ for some constant $0<c<1$.
For short, we use $\scr H$ to denote the sample space of the distribution, i.e.~the set of all possible outcomes of $H$.
Fix a constant $0 \le \eps < \min\{1, 1/c-1\}$, and define
\[
\hat{f} = \hat{f}(n) :=
\begin{cases}
2^{-n/m} \sqrt{pn(1-\eps)(1-c(1+\eps))} = \Omega(2^{-n/m} \sqrt{dn/m}) & \text{if } m=O(n),
\\
\sqrt{pn\log\gamma(1-\eps)(1-c(1+\eps))} = \Omega(\sqrt{pn\log\gamma}) & \text{if } m\gg n,
\end{cases}
\]
where $\gamma=\min\{pn,m/n\}$. Let $\kappa>0$ be a sufficiently small constant.
To prove Theorem~\ref{thm:edge_dependent_lower_bound}, it suffices to show that \whp\ $\disc(H) \ge \max \{ \kappa \hat f, 1\}$.
Note that this is trivially true when $\hat f \le 1/\kappa$, in view of Proposition~\ref{prop:edge_dependent_lower_bound_constant}. So we will assume that $\hat f > 1/\kappa$, and show that \whp\ $\disc(H) \ge  \kappa \hat f$.
Note that this assumption ensures that $n = O(m\log d)$, i.e.~there are not {\em too} many more vertices than edges.

Let $Z$ be the number of colourings $\psi:V\to\{-1,1\}$ with discrepancy $\disc(\psi) \le \kappa \hat f$.
We would like to prove an analogue of~\eqref{eq:EZ} in order to bound $\ex Z$, but unfortunately the random edges $e_1,\ldots,e_m$ of $H$ are no longer {\iid}
In order to overcome this obstacle, we introduce some random variables that will play an essential role in the analysis of $Z$.
For each $j=1,\ldots,m$ and $k=1,\ldots,n$, let $A_{j,k} := \bm{1}_{[v_k \in e_{j}]}$ be the $\{0,1\}$ value of the $(j,k)$ entry of the incidence matrix $\bm{A}$ of $H$, and let $\bm{A}_{j}=(A_{j,1},\ldots ,A_{j,n})$ denote the $j$-th row of $\bm{A}$.
Moreover, for each $k=1,\ldots,n$, we define
\begin{equation}\label{eqn:matrix_random_variables}
B_{0,k} := d
\qquad\text{and}\qquad
B_{i,k} := d - \sum_{j=1}^{i} A_{j,k}
\quad\text{for } i=1,\ldots,m.
\end{equation}
In other words, $B_{i,k}$ counts the number of ones that appear in the $k$-th column of $\bm{A}$ below the $i$-th row (recall that each column of $A$ has exactly $d$ ones). 
Note that each $B_{i,k}$ can be expressed as a function of the first $i$ rows of $\bm{A}$, and moreover the distribution of $A_{i+1,k}$ conditional on the outcome of $A_{1,k},\ldots,A_{i,k}$ can be described solely in terms of $B_{i,k}$.
More formally, for each $i=1,\ldots,m$, let $\scr F_i$ be the sigma algebra generated by $\bm{A}_1,\ldots,\bm{A}_i$, and let $\scr{F}_0:= \{ \emptyset, \scr H\}$ denote the trivial sigma algebra. Then, for each $i=0,\ldots,m$ and $k=1,\ldots,n$, $B_{i,k}$ is measurable with respect to $\scr F_i$, and (for $i<m$)
\[
P_{i,k} := \mb{P}[ A_{i+1,k} = 1 \mid \scr{F}_{i}] = \frac{B_{i,k}}{m-i}.
\]
Intuitively speaking, we would like that the above conditional probabilities remain close to $p$ (on average) as we reveal new rows of $\bm{A}$, at least for a large number of rows.
In view of that, for each $i=0,\ldots,m-1$, we consider the event $Q_i$ that for every $0\le j\le i$
\[
\sum_{k=1}^{n} P_{j,k} \ge (1 - \eps) pn
\qquad\quad\text{and}\qquad\quad
P_{j,k} \le (1+\eps) c
\quad\text{for $k=1,\ldots,n$}.
\]
(Here $(1+\eps)c<1$ from our choice of $\eps$.)
Observe that $\scr H=Q_0 \supseteq\cdots \supseteq Q_{m-1}$ is a decreasing sequence of events, and each $Q_i$ is $\scr{F}_i$-measurable by construction.
Let $\alpha := \max\{n/(n+m),1/2\}$.
We need the following technical result, which we prove in Section~\ref{sec:typical_histories}.
\begin{prop} \label{prop:typical_histories}
Under the assumptions in the proof of Theorem~\ref{thm:edge_dependent_lower_bound}, \lv{event} $Q_{\lfloor\alpha m\rfloor}$ holds \whp
\end{prop}
Now let $\Psi$ be the set of all colourings $\psi:V\to\{-1,1\}$, and pick an arbitrary $\psi\in\Psi$.
For each $i=1,\ldots,m$, let $R^\psi_i$ denote the event that $|\psi(e_i)| \le \kappa\hat{f}$, and let $R^\psi_{\le i}:=\bigcap_{j=1}^i R^\psi_j$. (By convention, $R^\psi_{\le0}=\scr H$.) Clearly, $R^\psi_i$ and $R^\psi_{\le i}$ are $\scr{F}_i$-measurable.
Note that, conditional upon any outcome of $\bm A_{1},\ldots,\bm A_{i-1}$ satisfying $Q_{i-1}$, the random variable $\psi(e_i)$ is distributed as $S_{\bm a,\bm p}$ in Section~\ref{sec:prelim} (with $a_k=\psi(v_k)$ and $p_k=\Pr[v_k\in e_i]$) and it satisfies the conditions of Lemma~\ref{lem:low_discrepancy_probability} (with $\zeta=(1+\eps)c<1$ and $[L,R]=[-\kappa\hat f,\kappa\hat f]$).
Hence, we can use that lemma to bound the conditional probability of $R^\psi_i$.
We first consider the sparse regime of $m=O(n)$.
By Lemma~\ref{lem:low_discrepancy_probability}, assuming $\kappa < 1/\left(3(c_{uni}+\sqrt{2/\pi})\right)$ and since $\kappa\hat f>1$,
\begin{equation}\label{eq:PRi1}
\mb{P}[ R^\psi_i \mid \scr{F}_{i-1}] \bm{1}_{[Q_{i-1}]}
\le \frac{c_{uni} + 2\kappa\hat f/\sqrt{2\pi}}{\sqrt{(1-\zeta)(1-\eps)pn}}
\le \frac{\kappa\hat f(c_{uni}+\sqrt{2/\pi})}{\sqrt{(1-\zeta)(1-\eps)pn}}
< 2^{-n/m} / 3.
\end{equation}
In particular, since $R^\psi_{\le i-1}\cap Q_{i-1}$ is $\scr{F}_{i-1}$-measurable and is contained in $Q_{i-1}$,
\[
\mb{P}[ R^\psi_1 ] \le 2^{-n/m} / 3
\qquad\text{and}\qquad
\mb{P}[ R^\psi_i \mid R^\psi_{\le i-1}\cap Q_{i-1}] \le 2^{-n/m} / 3
\quad\text{for } i=2,\ldots,m.
\]
Thus, for each $i=2,\ldots,m$,
\[
\mb{P}[ R^\psi_{\le i} \cap Q_{i-1}] = 
\mb{P}[ R^\psi_{i} \mid R^\psi_{\le i-1} \cap Q_{i-1}] \cdot \mb{P}[R^\psi_{\le i-1} \cap Q_{i-1}] \le
\left(2^{-n/m}/3\right) \mb{P}[R^\psi_{\le i-1} \cap Q_{i-2}],
\]
and inductively
\[
\mb{P}[ R^\psi_{\le i} \cap Q_{i-1}] \le \left(2^{-n/m}/3\right)^i.
\]
Let $t:=\lceil \alpha m\rceil$.
Next, we will bound $\disc(H)$ from below based on the discrepancies of the first $t$ rows of $\bm A$ when $Q_{t-1}$ holds.
First note that, since $t\ge nm/(n+m)$,
\begin{equation}\label{eq:PRi2}
\mb{P}[R^\psi_{\le t} \cap Q_{t-1}]
\le   \left(2^{-n/m}/3\right)^{t}
\le   \left(2^{-n/m}/3\right)^{nm/(n+m)}
= 2^{-n}  (2/3)^{nm/(n+m)} = o(2^{-n}),
\end{equation}
and then, by applying Markov's inequality to the random variable $Z\bm{1}_{ Q_{t-1}}$,
\begin{equation}\label{eq:PdiscQ}
\mb{P}[\text{$\disc(H) \le \kappa  \hat f$ and  $Q_{t -1}$}]
\le \mb{E}[Z  \bm{1}_{Q_{t-1}}]
= \sum_{\psi\in\Psi} \mb{P}[R^\psi_{\le m} \cap Q_{t-1}]
\le \sum_{\psi\in\Psi} \mb{P}[R^\psi_{\le t} \cap Q_{t-1}] = o(1).
\end{equation}
Before proceeding with the proof, we consider the dense regime of $m\gg n$. In that case, we obtain an analogue of~\eqref{eq:PRi1} by using the tighter inequality~\eqref{eqn:low_discrepancy_first_bound} in Lemma~\ref{lem:low_discrepancy_probability} instead of~\eqref{eqn:low_discrepancy_rough_bound}.
With $\zeta=(1+\eps)c<1$ and assuming that $C:=2\kappa^2/\pi<1/2$, we get
\begin{align*}
\mb{P}[ R^\psi_i \mid \scr{F}_{i-1}] \bm{1}_{[Q_{i-1}]}
&\le \frac{c_{uni}}{\sqrt{(1-\zeta)(1-\eps)pn}} + \left(1 - \exp\left(-\frac{(2\kappa\hat f)^{2}}{2\pi (1-\zeta)(1-\eps)pn} \right) \right)^{1/2}
\\
&= O(1/\sqrt{pn}) + \left(1 - \gamma^{- 2\kappa^2/\pi} \right)^{1/2}
=  1 - \Theta\left(\gamma^{-C} \right),
\end{align*}
where we used the fact that $1/\sqrt{pn}=o(\gamma^{-C})$.
Reasoning as before, we obtain the following analogue of~\eqref{eq:PRi2}:
\[
\mb{P}[ R^\psi_{\le t} \cap Q_{t-1}]
\le \left(1 - \Theta\left(\gamma^{-C} \right) \right)^t
\le \left(1 - \Theta\left(\gamma^{-C} \right) \right)^{m/2}
=  e^{ - n \Theta\left(\gamma^{-C}m/n \right)  }
= o(2^{-n}),
\]
where we used the facts that $t\ge m/2$ and $n/m=o(\gamma^{-C})$.
As a result, our bound in~\eqref{eq:PdiscQ} is also valid when $m\gg n$ as well.
Then, in either regime ($m=O(n)$ or $m\gg n$),
\begin{equation}\label{eq:almost_there}
\mb{P}[\disc(H) \le \kappa  \hat f]
\le \mb{P}[\text{$\disc(H) \le \kappa  \hat f$ and  $Q_{t -1}$}]
+ \mb{P}[\neg Q_{t -1}]
= o(1),
\end{equation}
by~\eqref{eq:PdiscQ}, Proposition~\ref{prop:typical_histories} and the fact
$Q_{\lceil \alpha m\rceil-1} \supseteq Q_{\lfloor\alpha m\rfloor}$. 
This shows that \whp\ $\disc(H) \ge \kappa\hat f$, and concludes the proof of Theorem~\ref{thm:edge_dependent_lower_bound}.
\end{proof}

\subsection{Proof of Proposition~\ref{prop:typical_histories}}\label{sec:typical_histories}

In this section, we prove Proposition~\ref{prop:typical_histories}.
For any $m\in\nat$, let $[m]:=\{1,2,\ldots,m\}$ and $[0]:=\emptyset$.
Suppose that $J \subseteq [m]$ is a fixed subset of size $0 \le j \le m$.
If $S \subseteq [m]$ is a random subset of size $d$,
then the distribution of the random variable $| S \cap J|$
is said to be \textbf{hypergeometric} with parameters $m,d$ and $j$. We
denote this distribution by $\text{Hyper}(m, d, j)$
in what follows. 
Now, $\text{Hyper}(m, d, j)$
is at least as concentrated about its expectation
as the binomial distribution, $\Bin(d, j/m)$  (see Chapter~21 in~\cite{Frieze2015} for details). As such, standard Chernoff
bounds ensure the following:
\begin{theorem}\label{thm:hypergeometric_concentration}
Suppose that $X \sim \text{Hyper}(m, d, j)$, and $\mu := \mb{E}[X] = d  j /m$. In this case, for every $0<\lambda<1$,
\[
    \mb{P}( |X - \mu| \ge \lambda  \mu) \le 2 \exp\left(\frac{- \lambda^2   \mu}{3} \right).
\]
\end{theorem}
Let $\bm A$ be the adjacency matrix of $H\sim\scr H(n,m,d)$, which has exactly $d$ ones in each column at random positions.
Let $p=d/m$.
Recall that, for each $i=0,\ldots,m$ and $k=1,\ldots,n$, the random variable $B_{i,k}$ counts the number of ones in the $k$-th column and below the $i$-th row of $\bm A$ (cf.~\eqref{eqn:matrix_random_variables}).
Also recall $P_{i,k}:=B_{i,k}/(m-i)$ (for $i<m$).
Clearly, $B_{i,k} \sim \text{Hyper}(m, d, m-i)$, so we may apply
Theorem~\ref{thm:hypergeometric_concentration} to control the value of $B_{i,k}$,
and thus of $P_{i,k}$.
Let $\alpha := \max\{n/(n+m),1/2\}$ and $t:=\lceil \alpha m\rceil$.
For a fixed column $k$, our goal is to show that \whp\ $B_{i,k}$ remains ``close'' to $\mb{E}[B_{i,k}] = d(m-i)/m$ for all $i=1,\ldots,t$.
By combining the error term in Theorem~\ref{thm:hypergeometric_concentration} with a na\"\i ve union bound, we can bound the probability of failure by something of the order of $m\exp(\Theta(-d))$,
which does not tend to $0$ unless $d = \Omega(\log m)$.
To overcome this, we need a more subtle argument in which we take the union bound over a smaller set of indices $i$ and take into account that $B_{i,k}$ does not change too much between two consecutive values of $i$.
This is made more precise in the following claim:
\begin{prop}\label{prop:first_history_choice}
Assume $0<\alpha,\lambda,\xi<1$ with $\xi\ge1/m$ and $\alpha+\xi<1$.
Fix $1 \le k \le n$. Then, with probability at least
\[
1 - 8\xi^{-1}  \exp\left(\frac{-d\lambda^2(1-\alpha-\xi)^2}{3} \right),
\]
it holds that, for all $i=0, \ldots, \lfloor \alpha m\rfloor$,
\begin{equation}\label{eq:Pconcentration}
(1-\lambda) \left(1 + \frac{\xi}{1-\alpha-\xi}\right)^{-1} p
\le P_{i,k} \le
(1+\lambda) \left(1 + \frac{\xi}{1-\alpha-\xi}\right) p.
\end{equation}
\end{prop}
\begin{proof}
As the columns of $\bm{A}$ are identically distributed, we may assume that $k=1$
in what follows. We thus drop the index $k$ from the notation of $A_{i,k},B_{i,k},P_{i,k}$ for simplicity.
Recall $B_i \sim \text{Hyper}(m, d, m-i)$ with $\ex B_i = p(m-i)$ for each $i=0,\ldots,m$.

Let $r:= \lceil (m-1)/\lfloor\xi m\rfloor\rceil$, which satisfies $1\le r\le m-1$ by assumption. Our first goal is to partition the set $[m-1]$ into $r$ intervals, each of size at most $\xi m$.
For each $q=0,\ldots,r-1$ let $I_q := [q \lfloor\xi m\rfloor]$, and let $I_r:= [m-1]$.
Then, setting $\til{I}_{q}:= I_{q} \setminus I_{q-1}$ for $q =1, \ldots, r$, gives us the desired partition $\til I_1,\ldots,\til I_r$ of $[m-1]$.
Now, let $r_0 := \lceil \lfloor\alpha m\rfloor/\lfloor\xi m\rfloor\rceil$.
Since $r_0 \lfloor\xi m\rfloor \ge \lfloor\alpha m\rfloor$, the set $[\lfloor\alpha m\rfloor]$ is contained in $\bigcup_{q=1}^{r_0}\til I_q$.
Clearly, $r_0\le r$ since $\lfloor\alpha m\rfloor\le m-1$. If $r_0=r$, then $(m-1) - \lfloor\alpha m\rfloor < \lfloor\xi m\rfloor$, which implies
$\lfloor\alpha m\rfloor + \lfloor\xi m\rfloor \ge m$ by integrality. This contradicts the fact that
$\lfloor\alpha m\rfloor + \lfloor\xi m\rfloor \le (\alpha+\xi)m < m$. As a result, $r_0\le r-1$.

For each $0\le q\le r-1$, define $Y_q:= B_{q\lfloor\xi m\rfloor}$ and let $Y_r:= B_{m-1}$. In other words, each $Y_q$ counts the number of ones in the $k$-th column of $\bm A$ below all the rows indexed by $I_q$.
We will prove that the variables $Y_0,\ldots,Y_{r_0}$ are concentrated around their mean, and from that derive a concentration result for $B_0,\ldots,B_{\lfloor\alpha m\rfloor}$.
Observe that $Y_r$ must be defined slightly differently, due to divisibility issues.
Fortunately, our argument will only require the analysis of $Y_0,\ldots,Y_{r_0}$, where $r_0\le r-1$, so this fact will cause no trouble.
For $q=0,\ldots,r-1$,
%
\[
\mb{E} Y_q= p(m - |I_{q}|) = p (m-q\lfloor\xi m\rfloor),
\]
%
and therefore, for every $q=1,\ldots,r_0$,
\begin{equation}\label{eq:EQrat}
\frac{\ex Y_{q-1}}{\ex Y_q}
= 1 + \frac{\lfloor\xi m\rfloor}{m-q\lfloor\xi m\rfloor}
\le 1 + \frac{\lfloor\xi m\rfloor}{m - \lfloor\alpha m\rfloor - \lfloor\xi m\rfloor}
\le 1 + \frac{\xi}{1-\alpha-\xi},
\end{equation}
where we also used the fact that
$r_0 \lfloor\xi m\rfloor \le \lfloor\alpha m\rfloor+\lfloor\xi m\rfloor$.
Now, let $E$ be the event that
\[
|Y_q - \ex Y_q| \le \lambda  \ex Y_q
\qquad\text{for all }q=0,\ldots,r_0.
\]
A direct application of Theorem~\ref{thm:hypergeometric_concentration} yields
\[
\mb{P}(\neg E)
\le \sum_{q=0}^{r_0} 2 \exp \left( \frac{-\lambda^2 p^2 (m-q\lfloor\xi m\rfloor)^2}{3d} \right)
\le 2(r_0+1) \exp \left( \frac{-\lambda^2 p^2 (m-r_0\lfloor\xi m\rfloor)^2}{3d} \right).
\]
Using the fact that $r_0 \lfloor\xi m\rfloor \le (\alpha +\xi) m$ and the rough bound
$r_0+1 \le \frac{\alpha m}{\xi m/2} + 2 \le \frac{4}{\xi}$,
we conclude that
\begin{equation}\label{eq:PnotE}
\mb{P}(\neg E)
\le (8/\xi) \exp \left( \frac{-d \lambda^2 (1-\alpha-\xi)^2}{3} \right).
\end{equation}
Finally, we turn our attention to $B_1,\ldots,B_{\lfloor\alpha m\rfloor}$.
For each $i \in [\lfloor\alpha m\rfloor]$, we pick $q\in [r_0]$ such that $i\in\til I_q$. Then, by monotonicity,
\[
Y_q \le B_i \le Y_{q-1}
\qquad\text{and}\qquad  
 \mb{E} Y_q \le \mb{E} B_i \le \mb{E} Y_{q-1}.
\]
Combining this with~\eqref{eq:EQrat} yields
\[
\ex Y_{q-1}
\le \left(1 + \frac{\xi}{1-\alpha-\xi}\right) \ex B_i
\qquad\text{and}\qquad
\ex Y_q
\ge \left(1 + \frac{\xi}{1-\alpha-\xi}\right)^{-1} \ex B_i.
\]
As a result, event $E$ implies that for every $1\le i\le\lfloor\alpha m\rfloor$,
\[
(1-\lambda) \ex Y_q \le
Y_q \le B_i \le Y_{q-1} \le (1+\lambda) \ex Y_{q-1},
\]
and hence
\begin{equation}\label{eq:Bconcentration}
(1-\lambda) \left(1 + \frac{\xi}{1-\alpha-\xi}\right)^{-1} \ex B_i
\le B_i \le
(1+\lambda) \left(1 + \frac{\xi}{1-\alpha-\xi}\right) \ex B_i.
\end{equation}
(Note that the equation above is also valid for $i=0$, since $B_0=d=\ex B_0$.)
Dividing~\eqref{eq:Bconcentration} by $m-i$, we conclude that event $E$ implies that, for every $0\le i\le\lfloor\alpha m\rfloor$,
\[
(1-\lambda) \left(1 + \frac{\xi}{1-\alpha-\xi}\right)^{-1} p
\le P_i \le
(1+\lambda) \left(1 + \frac{\xi}{1-\alpha-\xi}\right) p.
\]
Our bound on $\mb P(\neg E)$ in~\eqref{eq:PnotE} completes the proof of the proposition.
\end{proof}

\begin{corollary}\label{cor:typical_histories}
Suppose $m\ge d\to\infty$ and $n = O(m\log d)$ as $n\to\infty$.
Set $p=d/m$ and $\alpha = \max\{n/(n+m),1/2\}$.
Given any fixed constant $0<\eps<1$ and any $1\le k\le n$,
the following holds with probability at least
$1 - \exp\left(-\Omega(d/\log^3d)\right)$.
For every $i=0,\ldots,\lfloor\alpha  m\rfloor$,
\begin{equation}\label{eqn:corollary_typical_history}
|P_{i,k} - p| \le \eps  p.
\end{equation}
\end{corollary}
\begin{proof}
Since the probability bound in the statement is asymptotic as $n\to\infty$, we will implicitly assume throughout the proof that $n$ is sufficiently large for all the inequalities therein to be valid.
First, define $\lambda := \xi := 1/ \log^{3/2}d$. Clearly, $\xi \ge 1/d \ge 1/m$.
Observe that, since $n = O(m\log d)$, we have
\begin{equation}\label{eqn:1minusalpha}
1-\alpha = \min \{ m/(n+m), 1/2 \}  = \Omega(1/\log d).
\end{equation}
In particular $\alpha+\xi<1$, and thus all the assumptions in Proposition~\ref{prop:first_history_choice} are satisfied.
Moreover,
\[
\frac{\xi}{1-\alpha-\xi} = O\left(\frac{1/\log^{3/2}d}{1/\log d}\right) = O(\log^{-1/2} d).
\]
As a result, we can relax the inequalities in~\eqref{eq:Pconcentration} to
\[
P_{i,k} = \left(1 + O(\log^{-3/2}d) \right) \left( 1 + O(\log^{-1/2}d) \right) p = (1+o(1)) p,
\]
which implies that $|P_{i,k} - p| \le \eps  p$ (eventually for $n$ sufficiently large).
In view of Proposition~\ref{prop:first_history_choice}, this fails for some $i=0,\ldots,\lfloor\alpha  m\rfloor$ with probability at most
\[
8\xi^{-1}  \exp\left(\frac{-d\lambda^2(1-\alpha-\xi)^2}{3} \right)
\le 8(\log^{3/2}d) \exp\left(\frac{-d (1/2-\log^{-3/2}d)^2}{3\log^3 d} \right)
\]
\[
= \exp\left(-\Omega(d/\log^3 d) \right).
\]
This finishes the proof of the corollary.
\end{proof}

Now we are ready to prove Proposition~\ref{prop:typical_histories}, which we restate below in a more explicit form for convenience.
\begin{prop}
Let $0<\eps,c<1$ be fixed constants with $(1+\eps)c<1$. Assume that $d\to\infty$, $dn/m\to\infty$ and $n = O(m\log d)$ as $n\to\infty$, and suppose that $p=d/m\le c$. Let $\alpha = \max\{n/(n+m),1/2\}$.
Then \whp, for every  $i=0,\ldots,\lfloor\alpha m\rfloor$, 
\begin{equation}\label{eqn:lower_bound_average}
\sum_{k=1}^{n} P_{i,k} \ge (1-\eps) pn,
\end{equation}
and
\begin{equation} \label{eqn:column_upper_bound}
P_{i,k} \le (1+\eps)c
\quad\text{for } k=1,\ldots, n.
\end{equation}
\end{prop}

\begin{proof}
For $k=1,\ldots,n$, we say that column $k$ of $\bm{A}$ is \textit{controllable} if, for every $i=0,\ldots,\lfloor\alpha m\rfloor$, it holds
that
\[
    |P_{i,k}-p| \le \eps_0  p,
\]
where $\eps_0 := \eps/2$. Let $U\subseteq [n]$ denote the set of indices of the \textit{uncontrollable} columns. Then, by Corollary~\ref{cor:typical_histories} (with $\eps$ replaced by $\eps_0$),
\[
\mb{E} |U| \le  n  \exp\left(-\Omega(d/\log^3 d)\right) = o(n).
\]
Hence, we can apply Markov's inequality to ensure that $|U|/n \le \eps_0$ {\whp}
On the other hand, by applying the trivial lower bound to $P_{i,k}$ for each controllable column $k \in [n]$,
\begin{align*}
\sum_{k =1}^{n} P_{i,k}
\ge (n - |U|) (1- \eps_0)  p
\ge (1 - \eps_0)^2  pn
\ge (1 - 2\eps_0)  pn,
\end{align*}
\whp, thereby proving \eqref{eqn:lower_bound_average} (as $2  \eps_{0} = \eps$).

In order to verify that \eqref{eqn:column_upper_bound} holds, we first consider the regime in which $d \le \log^2 n$.
Observe then that \emph{deterministically}
\[
P_{i,k} \le \frac{d}{m-i} \le \frac{d}{(1 -\alpha)  m},
\]
for each $i=1,\ldots,\lfloor\alpha m\rfloor$ and $k=1,\ldots,n$. In particular,
since $1-\alpha = \Omega(\log^{-1} d)$ (in view of~\eqref{eqn:1minusalpha}) and $n=O(m\log d)$,
it holds that
\[
P_{i,k} = O(d(\log^2 d)/n) = o(1).
\]
Thus, \eqref{eqn:column_upper_bound} holds in this regime, as $c(1+\eps)>0$ is a fixed constant.
On the other hand, if $d \ge \log^2 n$, we can apply Corollary~\ref{cor:typical_histories} again, which ensures that with probability at least
\[
1 - n \exp\left(-\Omega(d/\log^3 d)\right) = 1-o(1)
\]
we have
\[
P_{i,k} \le (1 + \eps)  p \le (1 + \eps) c
\]
for all $i=0,\ldots,\lfloor\alpha m\rfloor$ and $k=1,\ldots,n$. The proof is therefore complete.
\end{proof}

\section{Upper Bounding Discrepancy---Proof of Theorem~\ref{thm:dense_setting}}\label{sec:upper-bound}

\sv{\toappendix{\section{Additions to Section~\ref{sec:upper-bound}}\label{app:color}}}


The main tool
we make use of is the algorithmic partial colouring lemma~\cite{Lovett_2012},
as done in~\cite{Bansal_2018,Potukuchi_2018,Potukuchi_2019}. For convenience,
we restate this lemma in the relevant hypergraph terminology, where
we define a \textbf{fractional colouring} to be a relaxation of a (hypergraph) colouring,
whose values are allowed to lie in the interval $[-1,1]$:

\begin{lemma}[Partial Colouring Lemma~\cite{Lovett_2012}] \label{lem:partial_colouring}
Suppose that $H=(V,E)$ is a hypergraph with $m$ edges and $n$ vertices which
are coloured by some  fractional colouring $\rho : V \rightarrow [-1,1]$.
Moreover, assume that $\delta > 0$ and $(\lambda_e)_{e \in E}$ are non-negative values such that
\lv{%
\begin{equation} \label{eqn:partial_colouring}
    \sum_{e \in E} \exp( - \lambda_{e}^{2}/16) \le \frac{n}{16}.
\end{equation}
}%
\sv{%
$    \sum_{e \in E} \exp( - \lambda_{e}^{2}/16) \le \frac{n}{16}.$
}%
Under these assumptions, there exists some fractional colouring $\psi: V \rightarrow [-1,1]$ for which
\begin{enumerate}
    \item $|\psi(e) - \rho(e)| \le \lambda_{e}  |e|^{1/2}$ for all $e \in E$, and
    \item $|\psi(v)| \ge 1 - \delta$ for at least $n/2$ vertices of $V$.
\end{enumerate}
Moreover, $\psi$ can be computed by a randomized algorithm in expected time
\lv{\[}
  \sv{$}
O((n+m)^{3}  \delta^{-2}  \log(m  n/ \delta)).
  \sv{$}
\lv{\]}
\end{lemma}
\begin{remark}
As previously mentioned, Levy, Ramadas, and Rothvoss~\cite{Levy2017DeterministicDM} showed
that the algorithmic guarantee of Lemma \ref{lem:partial_colouring} can be achieved using a deterministic algorithm (albeit with a worse run-time).
\end{remark}

In what follows, it will be convenient to refer to $\rho$ as the \textbf{target (fractional) colouring}
and $\delta$ as the \textbf{rounding parameter}.
%
We make use of Lemma~\ref{lem:partial_colouring} in the same way
as in~\cite{Bansal_2018, Lovett_2012, Potukuchi_2018, Potukuchi_2019}. In fact, we
analyze the same \textit{two-phase} algorithm considered by both Bansal and Meka~\cite{Bansal_2018}
and Potukuchi~\cite{Potukuchi_2018, Potukuchi_2019}, though we must tune
the relevant asymptotics carefully in order to achieve the bound claimed in Theorem~\ref{thm:dense_setting}.
In particular, we shorten phase one and change the target discrepancy in each application of Lemma~\ref{lem:partial_colouring}.
These modifications allow us to derive more precise asymptotics in the studied range of parameters.

Let us suppose that $H=(V,E)$ is a hypergraph drawn from
$\scr{H}(n,m,d)$ or $\mb{H}(n,m,p)$, where $p=d/m$. 
From now on,
we assume that $n$ is a power of $2$ for convenience. This follows
{\em w.l.o.g.}\ as we can always add extra vertices to $V$ which do not
lie in any of the edges of $G$.
Given the lower bounds of Theorems~\ref{thm:edge_independent_lower_bound}
and~\ref{thm:edge_dependent_lower_bound},
ideally we would like to compute an \textbf{output colouring}, $\phi:V \rightarrow \{-1,1\}$ 
with matching discrepancy. However, without finer proof techniques, this does not seem fully attainable for the full
parameter range of $m \gg n$.
Let us fix $\mu \df d n/m$.
We recall the definition of $\beta=\beta(n)$ as given in the statement of Theorem~\ref{thm:dense_setting}:
For all $n$ sufficiently large, $\beta(n) \ge 1$ and
\begin{equation} \label{eqn:beta_asymptotics_mu}
  \beta(n)  \mu \ge \log(m/n)  \left(\log\mu + 2 \right)^{5}.
\end{equation}

Let 
$\hat{f}:= \sqrt{
  \mu
 \log(m/n)  \beta}$ be the target upper bound on the discrepancy that we are aiming to prove.
Our argument is based on analyzing the \textsc{Iterated-Colouring-Algorithm}, which we now describe.
Fix $t_{1}\df\lg{\mu}$, $0 \le i \le t_{1}$, and let $H_0\df H$, and $\delta:=1/n$.
For convenience, we refer to the below procedure as \textbf{round} $i$.
We define $\hat{f}_{i}:= \hat{f}  (i+2)^{-2}$ as the desired discrepancy bound to be attained in round $i$.
\begin{enumerate}
    \item Remove all edges of $H_{i}=(V_i,E_i)$ of size less than or equal to $\hat{f}$.
    \item 
    Update $\lambda_{e}$ to be $\hat{f}_{i}/|e|^{1/2}$ for each $e \in E_{i}$. 
    Update the target colouring, which we denote by $\rho_{i}$, to be the previously computed colouring  \sv{$\psi_{i-1}$} \lv{$\psi_{i-1}:V_{i-1} \rightarrow [-1,1]$}  restricted to $V_{i}$
    (here $\psi_{-1}$ is the identically zero function by convention). 
    \item If $\sum_{e \in E_{i}} \exp(-\lambda_{e}^{2}/16) \le |V_{i}|/16$, 
    then apply Lemma~\ref{lem:partial_colouring} to $H_{i}$ with the above values, 
    yielding the fractional colouring \lv{$\psi_i: V_i \rightarrow [-1,-1]$}\sv{$\psi_i$}. Otherwise, abort the round
    and declare an error.
    \item 
      Assuming an error did not occur,
%
    compute $S_i \subseteq V_{i}$,
    such that $|\psi_{i}(v)| \ge 1 - 1/n$ for all $v \in S_i$, where $|S_i|:= |V_{i}|/2$. Afterwards, construct
    $H_{i+1}=(V_{i+1},E_{i+1})$ by restricting the edges of $E_{i}$ to $V_{i+1}$,
    where $V_{i+1} := V_i \setminus S_i$.
\end{enumerate}
We refer to the rounds $i=0, \ldots ,t_{1}$ as \textbf{phase one} of the algorithm's execution. 
Note that we refer to the vertices of $S_i$ as \textbf{inactive} after round $i$, as the value $\phi$ assigns to them will not change at any point onwards.
We refer to the remaining vertices as being \textbf{active}.
Observe then the following proposition:
\begin{prop} \label{prop:phase_one_bound}
For each $0 \le i \le t_{1}$ and $e \in E_i$, it holds that
\lv{\[}
  \sv{$}
  | \psi_{i}(e) - \rho_{i}(e) | \le \frac{\hat{f}}{(i+2)^2}. 
  \sv{$}
\lv{\]}
\end{prop}

Assuming that
none of the rounds yielded an error, there are exactly $2^{-t_{1}}  n = n  \mu^{-1} = m/d$ active vertices at the end
of phase one, as $t_1 = \lg \mu$. Heuristically, this means that we expect $H_{t_{1}+1}$ to have
edges of roughly constant size. As such, we can easily complete the colouring $\phi$ by executing the 
\textbf{phase two} procedure for rounds $t_{1} + 1 \le i \le t_{2}$, where $t_{2}:= \lg(10  n/\hat{f}) + 1$. 
The phase two procedure is identical to that of phase one, with the exception that in step 2, we update $\lambda_e$ to be $0$ (rather than $\hat{f}_{i}/|e|^{1/2}$) for each $e\in E_i$.
Analogously, we observe the following.
\begin{prop} \label{prop:phase_two_bound}
For each $t_{1} + 1 \le i \le t_{2}$ and $e \in E_i$, it holds that
\lv{\[}
  \sv{$}
    |\psi_{i}(e) - \rho_{i}(e)| =0.
  \sv{$}
\lv{\]}
\end{prop}

Assuming that
none of the rounds yielded an error, there are exactly $2^{-t_{2}}  n = n  \hat{f}/20n = \hat{f}/20$ active vertices at the end
of phase two.
In order to complete the construction of $\phi$, we conclude
with a \textbf{post-processing phase}. That is,
we arbitrarily assign $-1$ or $1$ to any of the vertices which remain
active at the end of phase two. Finally,
we round each remaining fractional value assigned by $\phi$ to the nearest
integer within $\{-1,1\}$.

Let us assume that the above procedure succeeds in its execution on $H$;
that is, it does \textit{not} abort during any iteration in either phase one
or two. In this case, we conclude the proof by showing the next lemma.

\begin{lemma}\label{lem:no-fail}
  If neither phase one nor phase two fails then Theorem~\ref{thm:dense_setting} holds.
\end{lemma}







\begin{proof}
  For each $0 \le i \le t_{2}$, let us formally extend $\psi_{i}$
to all of $V$. That is, define $\psi_{i}(v):= 0$ for each $v \in V \setminus V_{i}$,
and keep $\psi_{i}$ unchanged on $V_i$. Moreover, do the same for the target colouring $\rho_{i}$.
%
Observe then that once phase two ends,
$\phi$ can be expressed as a sum of differences involving the partial colourings $(\psi_{i})_{i=0}^{t_{2}}$
and $(\rho_{i})_{i=0}^{t_{2}-1}$
Specifically,
\lv{\[}
\sv{$}
    \phi(v) = \sum_{i=0}^{t_{2}} (\psi_{i}(v) - \rho_{i}(v)).
\sv{$}
\lv{\]}
Let $t_e$ be the time when edge $e$ becomes smaller than $\hat{f}$ or $t_e=t_2$ if it never happens. 
After applying Propositions~\ref{prop:phase_one_bound} and~\ref{prop:phase_two_bound} 
we get that
\begin{align*}
|\phi(e)| &\le \hat{f} + \sum_{i=0}^{t_{e}}|\psi_{i}(e) - \rho_{i}(e)|  \lv{\\}
\lv{&}\le \hat{f} + \sum_{i=0}^{t_{1}} |\psi_{i}(e)- \rho_{i}(e)| + \sum_{i=t_{1}+1}^{t_2} |\psi_{i}(e)- \rho_{i}(e)| \\
              &\le\hat{f} + \sum_{i= 0}^{\infty} \frac{\hat{f}}{(i+2)^{2}} = O(\hat{f}).
\end{align*}

The post-processing phase cannot increase the discrepancy that $\phi$ attains
on any edge of $E$ by more than $\hat{f}$ for the remaining active vertices; as we already observed, there are at most $\hat{f}/20$ of them.
The rounding of inactive vertices increases the discrepancy by at most 1.
\end{proof}


\toappendix{
\subsubsection*{Bounding the Failure Probability}

First, we recall the following lemma proven in~\cite{Bansal_2018} by Bansal and Meka:
\begin{lemma}[Lemma~6 in~\cite{Bansal_2018}] 
Suppose that $H$ is generated from $\scr{H}(n,m,d)$ 
and $\bm{M}$ is a fixed $r \times \ell$ sub-matrix of the $m\times \ell$ incidence matrix $\bm{A}$ of $H$. 
If $s \ge 10 d \ell/m$,
and $B(r, \ell, s)$ corresponds to the event in which each row of $\bm{M}$ has at least $s$ 1's,
then 
\[
  \mb{P}[B(r,\ell,s)] \le \exp\left(- \frac{r  s \log((s m)/(d \ell))}{2}\right).
\]
\end{lemma}

While this lemma is stated for the case when $H$ is generated from $\scr{H}(n,m,d)$, the upper bound
on $\mb{P}[B(r,\ell,s)]$ is proven by instead bounding the probability of the analogous event
when $H$ is generated from $\mb{H}(n,m,p)$ for $p=d/m$. As such, this lemma extends to the edge
independent model. We now restate it in a form which will be more convenient for our purposes.
\begin{lemma} \label{lem:submatrices}
Suppose that $H$ is generated from $\scr{H}(n,m,d)$ or $\mb{H}(n,m,p)$ 
for $p=d/m$, whose incidence matrix we denote by $\bm{A}$. If $s \ge 10  d \ell/m$,
then define $Q(r, \ell, s)$ as the event in which there exists
an $r \times \ell$ sub-matrix of $\bm{A}$ in which each row has at least $s$ 1's.
In this case,
\[
    \mb{P}[Q(r,\ell,s)] \le \binom{m}{r}  \binom{n}{\ell}   \exp\left(- \frac{ r  s \log((s m)/(d \ell))}{2}\right).
\]
\end{lemma}
}

Using \lv{this lemma}\sv{Lemma 6 from~\cite{Bansal_2018}}, we can ensure that \whp\ \textsc{Iterated-Colouring-Algorithm} will not
abort during phase one (Proposition \ref{prop:phase_one_error}) or two (Proposition~\ref{prop:phase_two_error}) and thus conclude the proof of Theorem~\ref{thm:dense_setting}.
\sv{We postpone the precise statements as well as proofs of the propositions to Appendix~\ref{app:color}.}

\toappendix{
\begin{prop} \label{prop:phase_one_error}
If \textsc{Iterated-Colouring-Algorithm} inputs a hypergraph
drawn from $\scr{H}(n,m,d)$ or $\mb{H}(n,m,p)$ where $p=d/m$,
then \whp\ it does not abort during phase one, provided we assume that
$\mu=d  n/m \rightarrow \infty$
and $m \gg n$.
\end{prop}

\begin{proof}

Given $0 \le i \le t_{1}$, we say that round $i$ is \textbf{good}, provided
there are at most $n_{i}/17$ rows of $H_{i}$ whose size is
greater than $s_{i}:= \beta  \mu/16  (i+2)^{5}$, where $\beta$ satisfies \eqref{eqn:beta_asymptotics_mu}. Otherwise,
we say that the round is \textbf{bad}.
Recall that $t_1=\lg{\mu}$.

Now, if round $i$ is good,
then we claim that \textsc{Iterated-Colouring-Algorithm} does
\textit{not} abort in iteration $i$. To see this, 
it suffices to show that for $n$ sufficiently large
\[
    \sum_{e \in E_{i}} \exp\left(-\lambda^{2}_{e}/16 \right) \le n_{i}/16,
\]
where $n_{i}:= |V_{i}|= n/2^i$, $\hat{f}_{i}:= \hat{f}  (i+2)^{-2}$ and $\lambda_{e}:= \hat{f}_{i}/|e|^{1/2}$ for $e \in E_{i}$. Observe now that since the round is good, we get that
\begin{align*}
  \sum_{e \in E_{i}} \exp\left(-\lambda^{2}_{e}/16 \right) &= \sum_{\substack{e \in E_{i}: \\ |e| \le s_{i}}} \exp\left(-\lambda^{2}_{e}/16 \right) + \sum_{\substack{e \in E_{i}: \\ |e| > s_{i}}} \exp\left(-\lambda^{2}_{e}/16 \right)  \\
     &\le m  \exp\left(-\frac{\hat{f}_{i}^{2}}{16  s_{i}} \right) + n_{i}/17.
\end{align*}
On the other hand, since $\hat{f}:= \sqrt{\mu  \log(m/n)  \beta}$,
\begin{align*}
    m  \exp\left( - \hat{f}_{i}^{2}/16  s_{i} \right) &= m  \exp(- \log(m/n)  (i+2))    \\
                                                          &= m  \left(\frac{n}{m} \right)^{i+2} \\
                                                          &= n  \left(\frac{n}{m} \right)^{i+1} = o(n_{i})
\end{align*}
where the last line follows since $(n/m)^{i+1} \ll 2^{-i}$, as $n \ll m$.
Thus,
\[
\sum_{e \in E_{i}} \exp\left(-\lambda_{e}^2/16 \right) \le (1 +o(1))  \frac{n_{i}}{17} \le \frac{n_{i}}{16}.
\]
We now must show that \whp, all of the rounds are good.
Now, observe that if some round $1 \le i \le t_{1}$ is \textit{bad}, then there
exists an $(n_{i}/17) \times n_{i}$ sub-matrix of $\bm{A}$, say $\bm{M}$, in which each row of $\bm{M}$ has greater than
$s_{i}$ 1's. In fact, since $n_{t_{1}} \le n_{i}$ and $s_{t_{1}} \le s_{i}$, we can take a sub-matrix
of $\bm{M}$ (and thus of $\bm{A}$) in which each row has at least $s_{t_1}$ 1's, 
and whose size is $(n_{t_{1}}/17) \times n_{t_{1}}$. Thus,
we observe the following claim:
\begin{enumerate}
\item If a bad round occurs, then there exists a $(n_{t_{1}}/17) \times n_{t_{1}}$ sub-matrix of $\bm{A}$
in which each row has more than $s_{t_1}$ 1's.
\end{enumerate}
Let us define $Q(n_{t_1}/17, n_{t_1}, s_{t_1})$ as this latter event;
namely, that there exists an $(n_{t_{1}}/17) \times n_{t_{1}}$ sub-matrix of $\bm{A}$
in which each row has more than $s_{t_1}$ 1's. In order to complete
the proof, it suffices to show that \whp, $Q(n_{t_1}/17, n_{t_1}, s_{t_1})$ does not 
occur.
Recall $n_{t_1}= n/\mu=m/d$.  as the number of active vertices drops by exactly half in each round. It follows that 
\[
  s_{t_1}=\frac{\beta\mu}{16(t_1+2)^5}\ge\frac{\log{\frac{m}{n}}(\log{\mu}+2)^5}{16(\lg{\mu}+2)^5}\gg5= 10\frac{dn}{2\mu m}\ge10\frac{dn_{t_1}}{m}.
  \]
Thus, we can apply Lemma~\ref{lem:submatrices} to ensure that
\begin{align*}
    \mb{P}[Q(n_{t_1}/17, n_{t_1}, s_{t_1})] &\le \binom{m}{n_{t_1}/17}  \binom{n}{n_{t_1}}  \exp\left(- \frac{n_{t_1}  s_{t_1}}{34}  \log\left(\frac{s_{t_1}m}{dn_{t_1}}\right)\right) \\
                                            & \le \binom{m}{n_{t_1}}^{2} \exp\left(- \frac{n_{t_1}  s_{t_1}}{34}  \log{s_{t_1}}\right) \\
                                            & \le \left(\frac{m  e}{n_{t_1}} \right)^{2  n_{t_1}}  \exp\left(- \frac{n_{t_1}  s_{t_1}\log{s_{t_1}}}{34}\right) 
\end{align*}
where the inequalities follow since $m \gg n, \binom{m}{n_{t_1}} \le (m  e/n_{t_1})^{ n_{t_1}}$. 
Now,
\begin{align*}
    \left(\frac{m  e}{n_{t_1}} \right)^{2  n_{t_1}} &=
    \exp\left( 2  n_{t_1}  ( \log(m/n) + \log(n  e/n_{t_1})) \right)\\
                                                    &=    \exp\left( 2  n_{t_1}  ( \log(m/n) + \log(e\mu  )) \right),
\end{align*}
so
\[
  \mb{P}[Q(n_{t_1}/17, n_{t_1}, s_{t_1})] \le \exp\left( -2  n_{t_1}  \left( \frac{s_{t_1}\log{s_{t_1}}}{68}  - \log(m/n) - \log(e\mu) \right)\right).
\]
Thus, by our assumption \eqref{eqn:beta_asymptotics_mu} on $\beta$, we get that
\[
  s_{t_1} = \frac{\beta \mu}{ 16 \left(\lg{\mu} +2 \right)^{5}}
\ge \frac{\log{\frac{m}{n}}(\log{\mu}+2)^5}{16(\lg{\mu}+2)^5}
\ge \frac{1}{16}  \log(m/n),
\]
and
\[
  s_{t_1} = \frac{\beta \mu}{ 16 \left(\lg{\mu} +2 \right)^{5}}
\ge 
\frac{\mu}{ 16 \left(\lg{\mu} +2 \right)^{5}},
\]
as $\beta\ge 1$.
The proposition follows as 
\[
  \frac{\log{(m/n)}\log(\log{(m/n)}/16)}{16\cdot 68}\gg\log(m/n),
\]
and
\[
  \frac{\mu}{68\cdot  16 \left(\lg{\mu} +2 \right)^{5}}\gg \log{(\mu e)}.\qedhere
\]
\end{proof}

\begin{prop} \label{prop:phase_two_error}
If \textsc{Iterated-Colouring-Algorithm} inputs a hypergraph
drawn from $\scr{H}(n,m,d)$ or $\mb{H}(n,m,p)$ where $p=d/m$,
then \whp\ it does not abort in phase two,
provided $\mu=d  n/m \rightarrow \infty$
as $n \rightarrow \infty$ and $m \gg n$.
\end{prop}

\begin{proof}

Suppose that $t_{1} +1 \le i \le t_{2}$ for
$t_{2}:= \lg(10  n/\hat{f}) + 1$.
Recall $n_{t_2}= n/2(10n/\hat{f})=\hat{f}/20$ and
that during phase two, $\lambda_{e}=0$ for each $e \in E_{i}$.
Thus, we get that
\[
    \sum_{e \in E_{i}} \exp\left(\lambda_{e}^{2}/16\right) = |E_{i}|.
\]
On the other hand, in order for an edge of $H$ to remain in $E_{i}$, it must have greater than 
$\hat{f}$ vertices which lie in $V_{i}$. As a result, if \textsc{Iterated-Colouring-Algorithm} aborts
in round $i$, then $H_i$ must have at least $n_{i}/16$ edges of size greater than $\hat{f}$. In particular, since
$n_{t_{2}} \le n_{i}$, this implies that the incidence matrix $\bm{A}$ of $H$ has an $(n_{t_{2}}/16) \times n_{t_{2}}$ sub-matrix in which
each row has greater than $\hat{f}$ 1's. Thus, if $Q(n_{t_2}/16, n_{t_2}, \hat{f})$ corresponds to the
event in which $\bm{A}$ has an $(n_{t_{2}}/16) \times n_{t_{2}}$ sub-matrix in which each row has greater than $\hat{f}$ 1's,
then we get the following claim:
\begin{enumerate}
    \item If \textsc{Iterated-Colouring-Algorithm} aborts in some round $t_{1} + 1 \le i \le t_{2}$, then
    $Q(n_{t_2}/16, n_{t_2}, \hat{f})$ must occur.
\end{enumerate}
As a result, in order to show that \whp\ \textsc{Iterated-Colouring-Algorithm} does \textit{not} abort in any round
it suffices to prove that $Q(n_{t_2}/16, n_{t_2}, \hat{f})$ does not occur {\whp}
Now, it follows that 
\[
  \hat{f} \ge10\frac{d\hat{f}}{20m} 
 =10\frac{dn_{t_2}}{m},
\]
as $d\le m$. Thus, we can apply Lemma~\ref{lem:submatrices} to ensure that
\begin{align*}
  \mb{P}[ Q(n_{t_2}/16, n_{t_2}, \hat{f})] 
  &\le  \binom{m}{n_{t_2}/16} \binom{n}{n_{t_2}}  \exp\left(- \frac{n_{t_2}  \hat{f}}{32}  \log\left(\frac{\hat{f}m}{dn_{t_2}}\right)\right)\\
  &\le  \binom{m}{n_{t_2}/16} \binom{n}{n_{t_2}}  \exp\left(- \frac{n_{t_2}  \hat{f}}{32}  \log\left(\frac{20m}{d}\right)\right),
\end{align*}
and so $\mb{P}[ Q(n_{t_2}/16, n_{t_2}, \hat{f})]$ is upper bounded by
\begin{equation}\label{eqn:final_probability_phase_two}
     \exp\left(-2  n_{t_2} \left( \frac{\hat{f}}{64}  - \log(m/n) - \log(\mu e)\right) \right),
\end{equation}
after applying the same simplifications as in Proposition~\ref{prop:phase_one_error}.
The proposition then follows by assumption (\ref{eqn:beta_asymptotics_mu}) on $\beta$, as 
\[
  \hat{f}/64=\sqrt{\beta\mu \log{(m/n)}}/64 \ge \frac{\log{(m/n)}\log^5{(\mu+2)}}{ 64}\gg\log{m/n},
\]
and
\[
  \frac{\log{(m/n)}\log^5{(\mu+2)}}{ 64}\gg \log{(\mu e)}.\qedhere
\]
\end{proof}
}

\section{Conclusion and Open Problems}
We have lower bounded the discrepancy
of the random hypergraph models $\mb{H}(n,m,p)$ and $\scr{H}(n,m,d)$ for the full parameter range in which $d \rightarrow \infty$ and $d n/m \rightarrow \infty$ where $p=d/m$. In the dense regime of $m \gg n$, we have provided asymptotically matching upper bounds, under the assumption that  $d = pm  \ge (m/n)^{1+\eps}$ for some constant $\eps>0$. These upper bounds are algorithmic, and so the main question left open by our work is whether analogous upper bounds can be proven in the sparse regime of $n/\log n \ll m \ll n$. Our lower bounds suggest that the discrepancy is $\Theta\left(2^{-n/m} \sqrt{p n} \right)$,
and while we believe that a second moment argument could be used to prove the existence of such a colouring---particularly, in the edge-independent model $\mb{H}(n,m,p)$---the partial colouring lemma does not seem to be of much use here. This leaves open whether such a colouring can be computed efficiently in this parameter range. If this is not possible, then ideally one could find a reduction to a problem which is believed to be hard on average. One candidate may be the random-lattice problem of Ajtai~\cite{Ajtai1996} and Goldreich et al.~\cite{Goldreich2011}, in which a random $m$ by $n$ matrix $\bm{M}$ with \iid\ entries from $\mb{Z}_{q}$ is generated,
and one wishes to compute a vector $\bm{x} \in \{0,1\}^{n}$ such that $\bm{M} \bm{x} =0$.

\lv{
\section*{Acknowledgements}
This work was initiated at the 2019 Graduate Research Workshop in Combinatorics, which was supported in part by NSF grant \#1923238, NSA grant \#H98230-18-1-0017, a generous award from the Combinatorics Foundation, and Simons Foundation Collaboration Grants \#426971 (to M.~Ferrara), \#316262 (to S.~Hartke) and \#315347 (to J.~Martin). We thank Aleksandar Nikolov for suggesting
the problem, and Puck Rombach and Paul Horn for discussions and encouragements in the early stages of the project.
Moreover, T.~Masa\v{r}\'ik received funding from the European Research Council (ERC) under the European Union’s Horizon 2020 research and innovation programme Grant Agreement 714704. He completed a part of this work while he was a postdoc at Simon Fraser University in Canada, where he was supported through NSERC grants R611450 and R611368.
X.~P\'erez-Gim\'enez was supported in part by Simons Foundation Grant \#587019.
}

\printbibliography

\appendix

\appendixText

\end{document}